\documentclass[11pt,oneside,reqno]{amsart}%
\usepackage{amsfonts}
\usepackage{amsmath}
\usepackage{amssymb}
\usepackage{graphicx}
\usepackage{epstopdf} 
\usepackage{version}
\usepackage{bm}
\usepackage{enumerate}
\usepackage{epsfig}
\usepackage{setspace}
\usepackage{caption}
\usepackage{color}
\usepackage{pgf,tikz,pgfplots}
\pgfplotsset{compat=1.15}

\DeclareCaptionLabelFormat{cont}{#1~#2\alph{ContinuedFloat}}

\theoremstyle{plain}
\newtheorem{theorem}{Theorem}[section]
\newtheorem{corollary}[theorem]{Corollary}
\newtheorem{lemma}[theorem]{Lemma}
\newtheorem{proposition}[theorem]{Proposition}

\newtheorem{assumption}[theorem]{Assumption}
\newtheorem{remark}[theorem]{Remark}

\newtheorem*{notation*}{Notation}
\numberwithin{equation}{section}
\excludeversion{comment1}

\let\<\langle
\let\>\rangle
\begin{document}
\title[]{Global and explicit approximation of piecewise smooth 2D functions from cell-average data}

\author{Sergio Amat}
\address{Departamento de Matemática Aplicada y Estadística. Universidad  Polit\'ecnica de Cartagena. Cartagena, Spain.}
\email{sergio.amat@upct.es}

\author{David Levin}
\address{School of Mathematical Sciences. Tel-Aviv University. Tel-Aviv, Israel.}
\email{levindd@gmail.com}

\author{Juan Ruiz-Alv\'arez}
\address{Departamento de Matemática Aplicada y Estadística. Universidad  Polit\'ecnica de Cartagena. Cartagena, Spain.}
\email{juan.ruiz@upct.es}

\author{Dionisio F. Y\'a\~nez}
\address{Departamento de Matem\'aticas, Universidad de Valencia. Valencia, Spain}
\email{dionisio.yanez@uv.es}

\maketitle

\begin{abstract}
Given cell-average data values of a piecewise smooth bivariate function $f$ within a domain $\Omega$, we look for a piecewise adaptive approximation to $f$. We are interested in an explicit and global (smooth) approach. Bivariate approximation techniques, as trigonometric or splines approximations, achieve reduced approximation orders near the boundary of the domain and near curves of jump singularities of the function or its derivatives. Whereas the boundary of $\Omega$ is assumed to be known, the subdivision of $\Omega$ to subdomains on which $f$ is smooth is unknown.
The first challenge of the proposed approximation algorithm would be to find a good approximation to the curves separating the smooth subdomains of $f$.
In the second stage, we simultaneously look for approximations to the different smooth segments of $f$, where on each segment we approximate the function by a linear combination of basis functions $\{p_i\}_{i=1}^M$, considering the corresponding cell-averages.
A discrete Laplacian operator applied to the given cell-average data intensifies the structure of the singularity of the data across the curves separating the smooth subdomains of $f$. We refer to these derived values as the signature of the data, and we use it for both approximating the singularity curves separating the different smooth regions of $f$. The main contributions here are improved convergence rates to both the approximation of the singularity curves and the approximation of $f$, an explicit and global formula, and, in particular, the derivation of a piecewise smooth high order approximation to the function.
\end{abstract}
\maketitle

\section{Introduction}

In some applications, e.g., in some numerical methods for solving flow problems, or in medical imaging, our raw data is in the form of cell-averages on a grid of cells. The underlying function $f$ is usually piecewise smooth, and the ultimate challenge is to find a high order piecewise smooth approximation to the function, including a high order approximation of the 'singularity' boundaries separating the domains of smooth pieces. This is the subject of the eminent paper \cite{ACDDM} by Ar\`andiga, Cohen, Donat, Dyn, and Matei, dealing with the 2D case, providing local approximations to the singularity curves, and local approximations of the underlying function $f$. The approach proposed is based on an edge-adapted nonlinear
reconstruction technique extending to the two-dimensional case
the essentially non-oscillatory subcell resolution technique introduced in
the one-dimensional setting by Harten \cite{Harten,ACDD}. It is
shown that for general classes of piecewise smooth functions, edge-adapted reconstructions
yield approximations with optimal rate of convergence.

The first step in \cite{ACDDM} is a detection and selection
mechanism generating a set of cells that are suspected
to be crossed by an edge and removing some cells
of this set to eliminate false alarms. The procedure is exact for step images.
In the present work, extending the ideas in \cite{ACDDM}, under some smoothness conditions, we derive an improved approximation to the singularity curves, improving the $O(h^2)$ approximation order in \cite{ACDDM} to an $O(h^3)$ approximation order. Our model is exact for piecewise linear functions with quadratic
edges. Furthermore, we introduce an $O(h^3)$ global implicit approximation of the singularity curves. We also suggest a simple approach for deriving piecewise-smooth and global approximations of the different smooth pieces of $f$. Combining both approximation results,
to the singularity curve and $f$, the Hausdorff distance between the graph of the function and the graph of the approximation is $O(h^3)$.

The paper is organized as follows: Section 2 recalls the cell-average data including the 2d case, a first approximation of the singularity curves is
described in Section 3, we derive a global approximation using the signature of the data, in Section 4 we propose a local high order approximation of the curve, this approximation is used to derive a final high order global approximation of the singularity curve. The full reconstruction algorithm is presented and analyzed theoretically in Section 5, in particular, a global, explicit, and piecewise-smooth approximation is obtained.

\section{Cell-average data}

\subsection{The univariate case}\hfill
\medskip

Let $f$ be a piecewise smooth function on $[0,1]$ with a jump discontinuity at $x=s$.
Given cell-average data
\begin{equation}\label{discretization}
 \bar{f}_j=\frac{1}{h}\int^{jh}_{(j-1)h}f(x)dx,\quad j=1, \cdots, N, \ \ h=1/N,
\end{equation}
we would like to construct a high order approximation to $f$, including a good approximation to the discontinuity location $s$.

Let us now define the sequence $\{F_j\}$ as
\begin{equation}\label{primitive}
F_j=h\sum_{i=1}^{j}\bar{f}_i=\int_{0}^{jh}f(y)dy.
\end{equation}
Consider the primitive of $f$, $F(x)=\int_0^xf(y)dy$, then the values $\{F_j\}_{j=0}^N$ represent a point-value discretization of $F$. Conversely, the cell-average values can be expressed in terms of the point values of the primitive through the expression,
\begin{equation}\label{diff}
\bar{f}_j=\frac{F_j-F_{j-1}}{h}.
\end{equation}
In \cite{ALR} we presented a non-linear subdivision algorithm that is adapted to the discontinuities, avoids the Gibbs phenomenon, attains the same regularity in smooth zones, as the equivalent linear algorithm, without diffusion and that presents $O(h^4)$ order of accuracy.

The key observation which stands at the base of the algorithm in \cite{ALR} is that the jump discontinuity in $f$ is reduced into a jump in the derivative of $F$ at $x=s$. From this, it follows that the location of the discontinuity can be obtained with high approximation order.

\subsection{The bivariate case}\hfill
\medskip

Let $f$ be a piecewise smooth function on $[0,1]^2$, defined by the combination of two pieces $f_1\in C^m[\Omega_1]$ and $f_2\in C^m[\Omega_2]$, $\Omega_2=[0,1]^2\setminus \Omega_1$. We further assume that the subdomains $\Omega_1$ and $\Omega_2$ are separated by a smooth curve $\Gamma$.

We assume that we are given cell-average values of $f$
\begin{equation}\label{2DAC}
 \bar{f}_{i,j}=\frac{1}{h^2}\int^{ih}_{(i-1)h}\int^{jh}_{(j-1)h}f(x)dxdy\quad i,j=1, \cdots, N, \ \ h=1/N,
\end{equation}
and we do not know the sets $\Omega_1$ and $\Omega_2$.
We would like to construct a high order approximation to $f$, including a good approximation to the separating curve $\Gamma$. All along the paper we assume the following conditions hold for $f_1$, $f_2$ and the separating curve $\Gamma$.

\begin{assumption}\label{extension}{\bf Regularity conditions}\hfill

\begin{enumerate}
\item  Both functions, $f_1$ and $f_2$, have $C^4$ extensions over $[-\varepsilon,1+\varepsilon]^2$ with $\varepsilon>0$.
\item The curve $\Gamma$ is a $C^4$ curve which is nowhere tangent to the boundary of $[0,1]^2$.
\item Let $min[f]_\Gamma$ be the minimal jump $|f_1-f_2|$ in $f$ across $\Gamma$, then $min[f]_\Gamma\ge \delta>0$.
\end{enumerate}
\end{assumption}

Following the 1-D case we may define the aggregated data sequence $\{F_{k,\ell}\}$ as
\begin{equation}\label{2Dprimitive}
F_{k,\ell}=h^2\sum_{i=1}^{k}\sum_{j=1}^{\ell}\bar{f}_{i,j}=\int_{0}^{\ell h}\int_{0}^{kh}f(x,y)dxdy.
\end{equation}

Consider the `2-D primitive' of $f$, $F(x,y)=\int_0^x\int_0^yf(s,t)dtds$, then the values $\{F_{k,\ell}\}_{k,\ell=0}^N$ represent point-value discretization of $F$. Conversely, the cell-average values can be expressed in terms of the point values of the primitive through the expression,
\begin{equation}\label{diff2D}
\bar{f}_{i,j}=\frac{1}{h^2}(F_{i,j}-F_{i-1,j}-F_{i,j-1}+F_{i-1,j-1}).
\end{equation}

As in the 1D case, the function $F$ is smoother than $f$, however,
unlike the 1-D case, the singularity structure of $F$ may be more complicated than that of $f$. That is, jumps in the derivatives of $F$ may occur across $\Gamma$ and across some additional lines tangent to $\Gamma$. Therefore, we adopt here a different approach for the approximation of $\Gamma$, directly using the cell-average data.

The main challenge is the detection and the approximation of $\Gamma$. A constructive local approach is presented in \cite{ACDD}, starting with the detection of cells that are crossed by $\Gamma$. This is done by adopting the ideas of Harten in \cite{Harten}, looking for high values in first-order differences in both $x$ and $y$ directions. In the next section, we suggest using the discrete Laplacian operator as an indicator of local singularities.

\section{Approximating the singularity curve}

\subsection{The signature of the data}\hfill

\medskip
Given the cell-average data ${\bf \bar{f}}=\{\bar{f}_{i,j}\}_{i,j=1}^N$, we define the signature of ${\bf \bar{f}}$ by applying a discrete Laplacian operator to the data:

\begin{equation}\label{sigf}
\sigma({\bf \bar{f}})=\{s_{i,j}\}_{i,j=2}^{N-1}\ ,
\end{equation}
where
\begin{equation}\label{sigf2}
s_{i,j}=\bar{f}_{i-1,j}+\bar{f}_{i,j-1}-4\bar{f}_{i,j}+\bar{f}_{i+1,j}+\bar{f}_{i,j+1},
\end{equation}
is the signature value attached to the cell
\begin{equation}\label{cij}
c_{i,j}=[(i-1)h,ih]\times[(j-1)h,jh],
\end{equation}
or to the center of the cell,
\begin{equation}\label{pij}
p_{i,j}=((i-1/2)h,(j-1/2)h).
\end{equation}
We would like to distinguish between regular cells, which are not crossed by $\Gamma$, and irregular cells which are crossed by $\Gamma$.
It turns out that the signature values are $O(h^2)$ as $h\to 0$ in areas where $f$ is $C^2$,
they are $O(h)$ where $f$ has a discontinuity in the first derivative, and they are $O(1)$ where $f$ has jump discontinuity.
Below we will use these properties of the signature values for both identifying and approximating the singularity curve $\Gamma$, and for approximating the function $f$.

Thus, heuristically, the big signatures are related to cells near the singularity curves.
For clarity, we choose to describe this fact alongside a specific numerical example.

\subsection{A numerical example}

Consider the piecewise smooth function on $[0,1]^2$ with a jump singularity across the curve $\Gamma$ which is the quarter circle defined by $x^2+y^2=0.5$. The test function is shown in Figure \ref{2Dtestf} and is defined as
\begin{equation}\label{testf2Dnonsmooth}
f(x,y)=
\begin{cases}
x+y& x^2+y^2 < 0.5,\\
1+0.5\sin(3y) & x^2+y^2\ge 0.5\\
\end{cases}
\end{equation}

\begin{figure}[!ht]
    \includegraphics[width=4in]{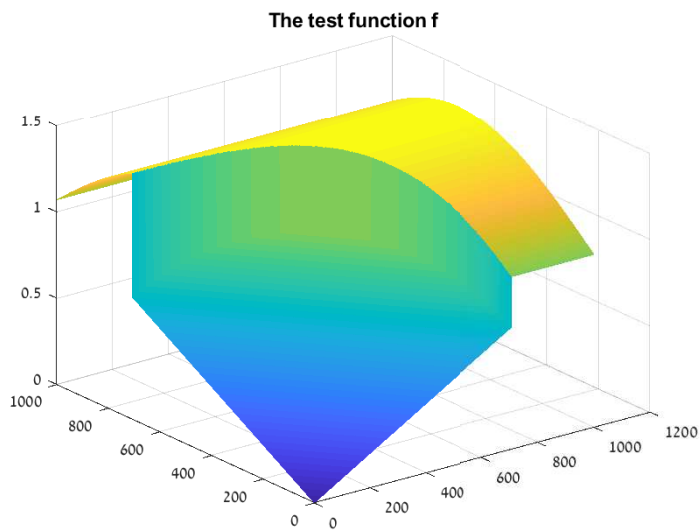}
    \caption{An example of a 2D non-smooth function}
    \label{2Dtestf}
\end{figure}

Next, in Figure \ref{2DtestfCA},  we present the $N\times N$ cell-average data derived from the above test function computed for $N=40$, $h=0.025$.

\begin{figure}[!ht]
    \includegraphics[width=4in]{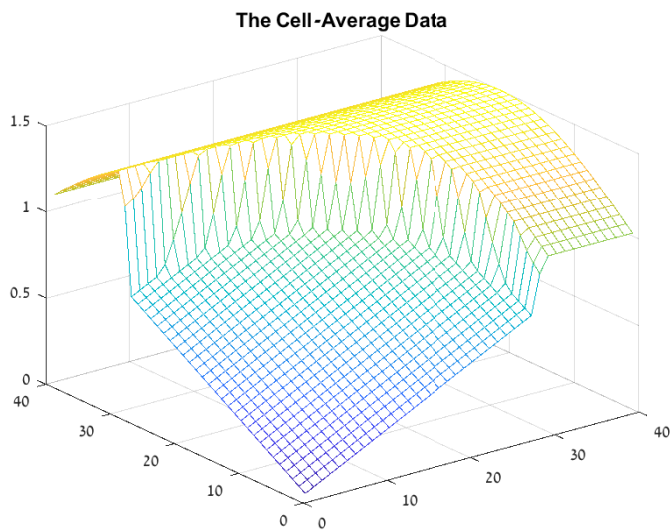}
    \caption{Cell-average data of a 2D non-smooth function}
    \label{2DtestfCA}
\end{figure}

In Figure \ref{sig2DtestfCA} we display the signature of the cell-average data. Large values are obtained at cells near the singularity curve.

\begin{figure}[!ht]
    \includegraphics[width=4in]{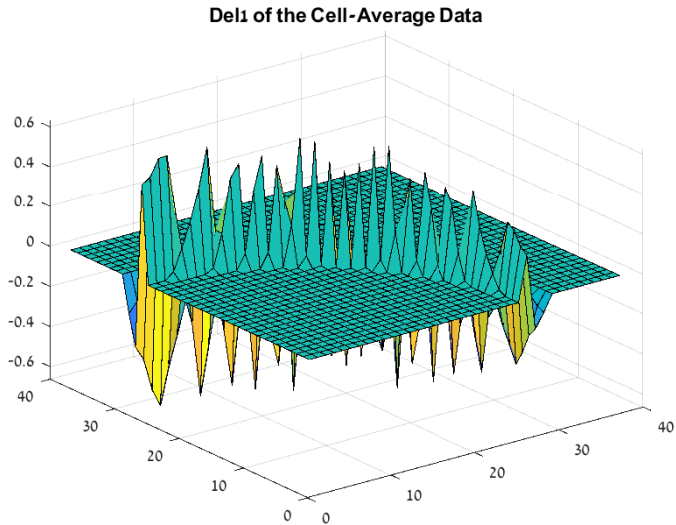}
    \caption{The signature of the Cell-average data of a 2D non-smooth function}
    \label{sig2DtestfCA}
\end{figure}

\begin{figure}[!ht]
    \includegraphics[width=4in]{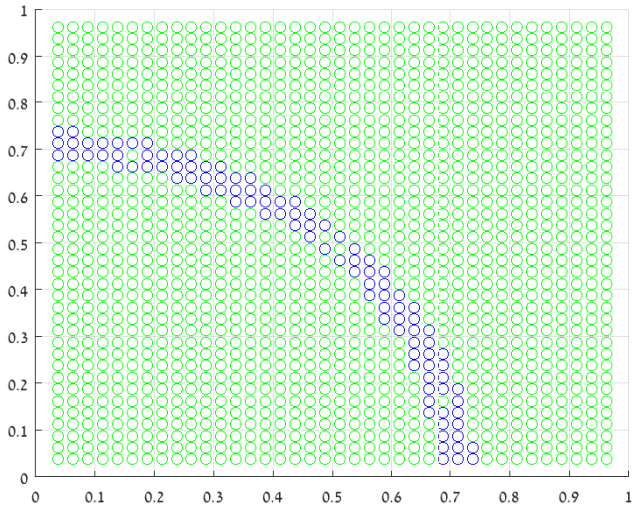}
    \caption{Cells with large signature values, in blue}
    \label{P1P2Q}
\end{figure}

We denote the cells with large signature values, i.e. $\ge 0.1 max(|\sigma({\bf \bar{f}})|)$ by $U_0$. These cells, depicted in blue in Figure \ref{P1P2Q}, serve us for sorting the rest of the cells, the green cells, into two sets $U_1$ and $U_2$ in either side of the singularity curve $\Gamma$. The sets $U_0$, $U_1$, and $U_2$ are considered both as sets of cells and as sets of midpoints of cells.

Now, we would like to propose an algorithm based on theoretical shreds of evidence, allowing us to develop a full theoretical analysis of our approach.

\medskip
By the definition of the signature operator, the following observations follow:
\begin{lemma}\label{Oh2dist}
Let $M=max(sup_{\Omega_1}\{|f_{1,xx}|+|f_{1,yy}|\}, sup_{\Omega_2}\{|f_{2,xx}|+|f_{2,yy}|\})$. Then for a cell $c_{i,j}$ centered at $p_{i,j}$, if $dist(c_{i,j},\Gamma)\ge h$
$$|s_{i,j}|\le Mh^2,$$
and if $h/2<dist(p_{i,j},\Gamma)\le h$
$$|s_{i,j}|\ge \delta_{i,j}/2$$
for a small enough $h$, where $\delta_{i,j}$ is the magnitude of the jump in $f$ across $\Gamma$ at the closest point to $p_{i,j}$.
\end{lemma}

\medskip
Recalling Assumption \ref{extension}(3), we suggest the following detection algorithm:

\medskip
{\bf The Detection Algorithm}

\medskip
We let $c_{i,j}\in U_0$ if

$$|s_{i,j}| \ge \delta/2.$$

Using Lemma \ref{Oh2dist} it follows that detection algorithm identifies cells which are near $\Gamma$ if $h<\kappa h_c$ where
$$h_c= \sqrt{\frac{\delta}{M}},$$
where $\kappa$ is small enough.

\medskip

In practice, we need an estimation of $\delta$. By definition, it is smaller than the minimal jump $|f_1-f_2|$ in $f$ across $\Gamma$.
The first step is to distinguish the jumps computed in smooth regions $U_1$ or $U_2$ from the jumps crossing the singularity curve.
Let $h^{'}_c$ small enough such that $$\max_{U_1 \bigcup U_2} \{ ||\nabla f||_1 \} h^{'}_c < \sqrt{h^{'}_c}$$ or equivalently $$\max_{U_1 \bigcup U_2} \{ ||\nabla f||_1 \}  < \frac{1}{\sqrt{h^{'}_c}}.$$ Then, we have
$|\bar{f}_{i,j}-\bar{f}_{i^{'},j^{'}}| < \sqrt{h^{'}_c}$ for
all $(i,j)$ and $(i^{'},j^{'})$ in a neighborhood of  $U_1$ or $U_2$. In particular, the jumps crossing $\Gamma$ will be greater than $\sqrt{h^{'}_c}$. Therefore, we can propose
$$\delta^*=\min \{ |\bar{f}_{i,j}-\bar{f}_{i^{'},j^{'}}|  \ s.t. \  |\bar{f}_{i,j}-\bar{f}_{i^{'},j^{'}}| > \sqrt{h^{'}_c}  \} $$
as an approximation of $\delta$. In fact, when the differences cross $\Gamma$ its absolute value is $O(1)$ (bigger than $\frac{\delta}{2} >\sqrt{h^{'}_c}$).

Then, we can take $h\leq \kappa \min \{ h_c, h^{'}_c\}$ where
$$\max_{U_1 \bigcup U_2} \{ ||\nabla f||_1 \}  < \frac{1}{\sqrt{h^{'}_c}}$$
and
$$
h_c= \sqrt{\frac{\delta^*}{M}}.
$$

\subsubsection{A global implicit approximation of $\Gamma$}\label{AppGamma}\hfill

\medskip
Similar to the approach in \cite{ALR2} we define the approximation to the singularity curve $\Gamma$ as the zero level curve of a bivariate spline function approximating a proper signed-distance function. Referring to the sets $U_1$, $U_2$ and $U_0$ defined above, we attach to each point $p\in U_1$ the value $dist(p,U_0)+h$ and to each point $p\in U_2$ the value $-dist(p,U_0)-h$. The addition of $\pm h$ is due to the observation that the `width' of the set $U_0$ is approximately $2h$, which is the size of the support of the discrete Laplacian operator defining the signature. Now we find a smooth bivariate spline on $[0,1]^2$ approximating the above data on $U_1$ and $U_2$ in the least-squares sense. The zero level curve of this spline is the suggested approximation to $\Gamma$.
Using a bi-cubic spline $S_3(x,y)$ with knots distance 0.25, we derive the approximation to $\Gamma$ presented in Figure \ref{AGamma}, where the exact $\Gamma$ is depicted in blue and the approximation in red. We observe that the approximation is deteriorating near the two ends of the curve. A better approximation, using the same parameters, is displayed in Figure \ref{AGamma2} for the case of a closed singularity curve $\Gamma_2$ defined by

\begin{equation}\label{Gamma2}
(x-0.5)^2+(y-0.5)^2=0.1.
\end{equation}

\begin{figure}[!ht]
    \includegraphics[width=4in]{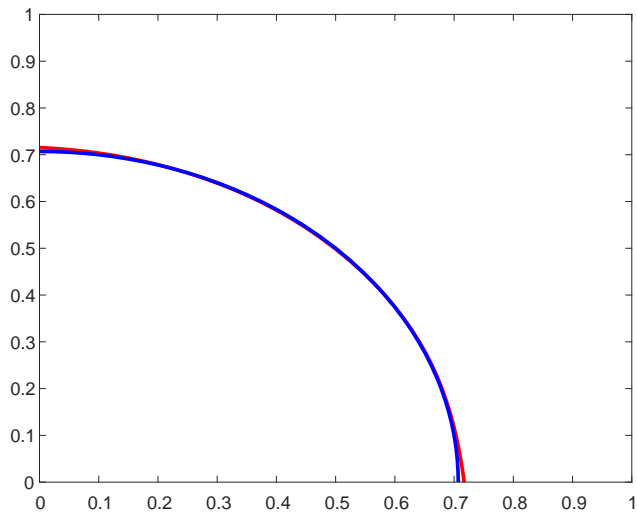}
    \caption{The singularity curve in blue and its approximation in red}
    \label{AGamma}
\end{figure}

\begin{figure}[!ht]
    \includegraphics[width=4in]{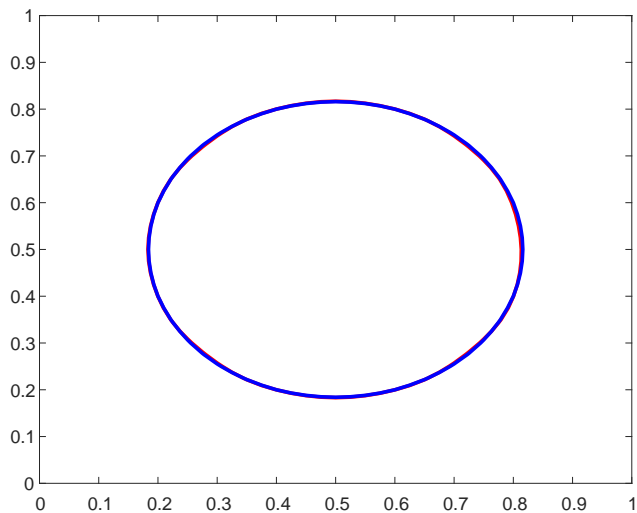}
    \caption{Approximation of a smooth closed singularity curve.}
    \label{AGamma2}
\end{figure}

The above procedure gives a global approximation of the singularity curve with an $O(h)$ approximation order. A local piecewise linear method for approximating the singularity curve is presented in \cite {ACDDM}, and this is already an $O(h^2)$ method. In the next section, we extend the ideas in \cite {ACDDM} to finding an $O(h^3)$ local approximations of the singularity curve.
Further on, we use the local approximations for deriving a global $O(h^3)$ approximation of the curve.

\section{Quadratic approximation to $\Gamma$.}

\bigskip
In this section, we introduce a quadratic approximation of the singularity curve.

\begin{figure}[!ht]
    \includegraphics[width=4in]{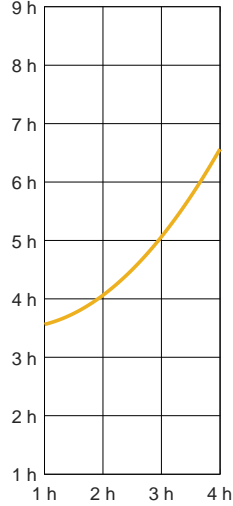}
    \caption{A zoom on a neighborhood of $\Gamma$}
    \label{gridplusquad}
\end{figure}

\medskip
In \cite{ACDDM} the authors present an algorithm that is exact for step images with a linear $\Gamma$. We suggest making an algorithm that is exact for an image that is linear on each side of $\Gamma$, and $\Gamma$ is a quadratic function. It seems that this is the minimal for obtaining an $O(h^3)$ approximation to $\Gamma$.

Assume the function $f$ below $\Gamma$ is
$$L_1(x,y)=\alpha_1x+\beta_1 y+\gamma_1,$$
and the function above $\Gamma$ is
$$L_2(x,y)=\alpha_2x+\beta_2 y+\gamma_2.$$
Assume also that $\Gamma$ is passing between $y=3h$ and $y=7h$ as in Figure \ref{gridplusquad} and is given by the equation
$$q(x)=ax^2+bx+c.$$

We are given the function cell-averages $\bar f_{ij}$ defined in (\ref{2DAC}). Using these values we would like to retrieve the parameters of the linear-quadratic model. The suggested algorithm is as follows:

\begin{itemize}
\item
First, we find $\alpha_1,\beta_1,\gamma_1$ by matching the cell-averages of the linear model to the given cell-averages $\bar f_{3,2}, \bar f_{3,3}, \bar f_{4,3}$.
This gives us a linear system for the 3 unknowns.
\item
Second, we find $\alpha_2,\beta_2,\gamma_2$ by matching the cell-averages of the linear model to the given cell-averages $\bar f_{3,9}, \bar f_{3,8}, \bar f_{4,8}$.
\item
Third, using the function parameters computed above, we look for the curve parameters $a,b,c$. This is done by matching the cell-averages of the linear model to the given cell-averages on three triplets of cells $\bar f_{2,4}+\bar f_{2,5}+ \bar f_{2,6}+ \bar f_{2,7}$, $\bar f_{3,4}+\bar f_{3,5}+ \bar f_{3,6}+ \bar f_{3,7}$, $\bar f_{4,4}+\bar f_{4,5}+ \bar f_{4,6}+ \bar f_{4,7}$. This matching gives a system of three quadratic equations in the unknowns $a,b,c$.
\end{itemize}

\subsection{The system of quadratic equations}

\begin{equation}\label{quadsystem}
Q_i(a,b,c)=\bar f_{i,4}+\bar f_{i,5}+\bar f_{i,6}+\bar f_{i,7},\ \ \ \ i=2,3,4.
\end{equation}

where
\begin{equation}\label{Qidef}
Q_i(a,b,c)\equiv h^{-2}\left(\int^{ih}_{(i-1)h}\int^{q(x)}_{3h}(L_1(x,y)-L_2(x,y))dydx+
 \int^{ih}_{(i-1)h}\int^{7h}_{3h}L_2(x,y)dydx\right).
\end{equation}

Denoting $L_1(x,y)-L_2(x,y)=\alpha x+\beta y+\gamma$, let us expand the first integral above:
$$\int^{ih}_{(i-1)h}\int^{q(x)}_{3h}(\alpha x+\beta y+\gamma)dydx=\int^{ih}_{(i-1)h}[(\alpha x+\gamma)(q(x)-3h)+\beta/2(q(x)^2-(3h)^2)]dx.$$

Approximating any other segment of $\Gamma$ can be done by transforming the coordinates so that the problem is defined in $[h,3h]\times [h,9h]$ as in Figure \ref{gridplusquad}.
Thus we look for $q(x)=ax^2+bx+c$ such that $q:[2h,3h]\ \to\ [3h,7h]$. Observing that $c=O(h)$ we may rewrite $q$ as
\begin{equation}\label{ch}
q(x)=ax^2+bx+ch,
\end{equation}
where now the unknown parameters $a$, $b$ and $c$ are $O(1)$ as $h\to 0$.

Since $q(x)=ax^2+bx+ch$, the above integral is a quadratic form in the unknowns $a,b,c$.

\subsection{Local approximation order}\hfill

\medskip
Let us assume that $\Gamma$ is locally defined by a $C^3$ function $g(x)$, and let the function $f$ be $f_1\in C^2$ below $\Gamma$ and $f_2\in C^2$ above $\Gamma$. Consider the same local neighborhood $R$ of $\Gamma$ described in Figure \ref{gridplusquad}. Within $R$ there exists a local quadratic approximation $\tilde q(x)\sim g(x)$, $\tilde q(x)=\tilde ax^2+\tilde bx+\tilde ch$ such that
\begin{equation}\label{yq}
g(x)=\tilde q(x)+O(h^3).
\end{equation}
Denoting the linear functions defined above as
$\tilde L_1(x,y)\sim f_1(x,y)$ and $\tilde L_2(x,y)\sim f_2(x,y)$,
they satisfy
$$f_1(x,y)=\tilde L_1(x,y)+O(h^2),$$
$$f_2(x,y)=\tilde L_2(x,y)+O(h^2),$$
in $R$, as $h\to 0$.
Denoting by $\tilde\Gamma$ the graph of $\tilde q$ in $R$,
by $\tilde L(x,y)$ the function which is $\tilde L_1$ below $\tilde\Gamma$ and
$\tilde L_2$ above $\tilde\Gamma$, let us compare the cell-averages $\bar f_{i,j}$ of $f$ and $\bar L_{i,j}$ of $\tilde L$ within $R$. Since the area between $\Gamma$ and $\tilde\Gamma$ on a cell of size $h\times h$ is $O(h^4)$, it follows that
\begin{equation}\label{fLh2}
\bar f_{i,j}=\bar L_{i,j}+O(h^2).
\end{equation}

Let us consider the system of equations for the quadratic model for $\Gamma$. We know the system is solvable if we replace all the values $\{\bar f_{i,j}\}$ used for cell-average matching by the corresponding values $\{\bar L_{i,j}\}$. By (\ref{fLh2}), it follows that the system (\ref{quadsystem}) is an $O(h^2)$ perturbation of the solvable system
\begin{equation}\label{quadsystemL}
Q_i(\tilde a,\tilde b,\tilde c)=\bar L_{i,4}+\bar L_{i,5}+\bar L_{i,6}+\bar L_{i,7},\ \ \ \ i=2,3,4.
\end{equation}

Using this observation we would like to show that the solution of the system (\ref{quadsystem}) is an $O(h^2)$ perturbation of the solution of (\ref{quadsystemL}). Let us denote the mapping $Q:\mathbb{R}^3\to \mathbb{R}^3$, $Q(a,b,c)=(Q_1(a,b,c),Q_2(a,b,c),Q_3(a,b,c))^t$, the right hand side of the system (\ref{quadsystem}) as $\bar F\in \mathbb{R}^3$, and  the right hand side of the system (\ref{quadsystemL}) as $\bar L\in \mathbb{R}^3$.
To solve the system
\begin{equation}\label{QminusF}
Q(a,b,c)-\bar F=0,
\end{equation}
we use Newton's iterations. Since $\bar F$ is a constant vector, the iterations are
\begin{equation}\label{Newtons}
(a^{[n+1]},b^{[n+1]},c^{[n+1]})^t=(a^{[n]},b^{[n]},c^{[n]})^t-J_Q(a^{[n]},b^{[n]},c^{[n]})^{-1}(Q(a^{[n]},b^{[n]},c^{[n]})-\bar{F}),
\end{equation}
$n\ge 0$, where $J_Q$ is the Jacobian of $Q$.
\begin{lemma}\label{Lemma4.1}
Under condition \ref{extension}(3), for a small enough $h$,  $\|J_Q^{-1}\|_\infty\ge C_J>0$.
\end{lemma}
\begin{proof}
It is enouth to show that $|det(J_Q)|\ge\beta>0$ for $\beta>0$, where $\beta$ is independent upon the approximation location.
For the case of constant functions $L_1=c_1$, $L_2=c_2$, it follows, by direct calculations, that
$$det(J_Q)=2(c_2-c_1)^3.$$
For the case of linear functions $L_1$, $L_2$, in view of condition \ref{extension}(3), it follows by a perturbation argument that
$$|det(J_Q)|\ge 2\delta^3+O(h),$$
as $h\to 0$.
\end{proof}

\begin{proposition}\label{PropOh2}
Let $\tilde v^{[0]}=(\tilde a,\tilde b,\tilde c)^t$ and let $v^{[k]}=(a^{[k]},b^{[k]},c^{[k]})^t$ be the vectors defined by Newton's iteration (\ref{Newtons}). Assuming the Hessian of $Q$ satisfies $\|H_Q\|_\infty\le C_H$, and $\|\bar L-\bar F\|_\infty\le Ch^2$, then, for a small enough $h$, Newton's iterations converge to a solution $v^{[\infty]}$ of (\ref{QminusF}),
and
\begin{equation}\label{Qvkd2}
v^{[\infty]}=\tilde{v}^{[0]}+O(h^2),
\end{equation}
as $h$ tends to $0$.
\end{proposition}

\begin{proof}
First we prove that for a small enough $h$
\begin{equation}\label{Qvkd}
\|Q(v^{[k]})-\bar F\|_\infty \le C^{2^k}C_J^{2^{k+1}-2}C_H^{2^k-1}h^{2^{k+1}}, k\ge 1.
\end{equation}
The proof is by induction.
Starting with $\tilde v^{[0]}$, as $Q(\tilde v^{[0]})=\bar L$, by hypothesis
\begin{equation}\label{eq1}
\|Q(\tilde v^{[0]})-\bar F\|_\infty=\|\bar L-\bar F\|_\infty\le Ch^2,
\end{equation}
and by \eqref{Newtons}:
$$v^{[1]}=\tilde v^{[0]}-J_Q(\tilde{v}^{[0]})^{-1}(Q(\tilde v^{[0]})-\bar F).$$
Using Lemma \ref{Lemma4.1}, it follows that  $J_Q(\tilde a, \tilde b,\tilde c)^{-1}$ exists and is smooth and bounded in a neighborhood of $(\tilde a, \tilde b,\tilde c)$, with a bound which is independent upon $h$. Then by \eqref{eq1} it follows that
$$\|v^{[1]}-\tilde v^{[0]}\|_\infty\le CC_Jh^2$$
for a small enough $h$.

Assuming \eqref{Qvkd} holds for $k=n$, from Newton's iteration formula (\ref{Newtons}) it follows that
\begin{equation*}
\begin{split}
\|v^{[n+1]}-v^{[n]}\|_\infty&\le \|J_Q(v^{[n]})^{-1}(Q(v^{[n]})-\bar F)\|_\infty\\
&\le C^{2^n}C_J^{2^{n+1}-1}C_H^{2^n-1}h^{2^{n+1}},
\end{split}
\end{equation*}
for $h<0.5 (CC_J^2C_H)^{-1}$ we obtain
\begin{equation}\label{fast1}
\|v^{[n+1]}-v^{[n]}\|_\infty\le C_J^{-1}C_H^{-1}2^{-2^n}h^2.
\end{equation}
The fast decay implies
\begin{equation}\label{fastsum}
\|v^{[n+1]}-\tilde v^{[0]}\|_\infty\le \tilde C \sum_{k=1}^n2^{-2^k} h^2,
\end{equation}
i.e. $v^{[n+1]}$ is in a neighborhood of $\tilde v^{[0]}$. If we expand $Q(v^{[n+1]})$ in Taylor series about $v^{[n]}$:

\begin{equation}\label{Taylors2}
\begin{split}
Q(v^{[n+1]})=&Q\big(v^{[n]}-J_Q(v^{[n]})^{-1}(Q(v^{[n]})-\bar F)\big) \\
=&Q(v^{[n]})-J_Q(v^{[n]})J_Q(v^{[n]})^{-1}(Q(v^{[n]})-\bar F) \\
&+(J_Q(v^{[n]})^{-1}(Q(v^{[n]})-\bar F))^tH_Q(J_Q(v^{[n]})^{-1}(Q(v^{[n]})-\bar F)) \\
=&\bar F +(J_Q(v^{[n]})^{-1}(Q(v^{[n]})-\bar F))^tH_Q(J_Q(v^{[n]})^{-1}(Q(v^{[n]})-\bar F)) .
\end{split}
\end{equation}

Using the induction hypothesis, it follows that
\begin{equation}\label{Qvnp1}
\begin{split}
\|Q(v^{[n+1]})-\bar F\|_\infty &\le C_J^2C_H (C^{2^n}C_J^{2^{n+1}-2}C_H^{2^n-1}h^{2^{n+1}})^2 \\
&=C^{2^{n+1}}C_J^{2^{n+2}-2}C_H^{2^{n+1}-1}h^{2^{n+2}}.
\end{split}
\end{equation}
Finally, the fast decay observed in \eqref{fast1} also implies that the sequence $\{v^{[n]}\}$ is a Cauchy sequence, and therefore converging. By \eqref{Qvnp1}, it also follows that $v^{[\infty]}=\lim_{n\to\infty}v^{[n]}$ solves the quadratic system, and the result (\ref{Qvkd2}) follows from \eqref{fastsum}.
\end{proof}

From Proposition \ref{PropOh2}, it follows that the solution of the quadratic system for the unknown coefficients $a$, $b$ and $c$ of $q$ is an $O(h^2)$ perturbation of the corresponding coefficients of $\tilde{q}$. Therefore, in view of (\ref{ch}), in the interval $[2h,3h]$,
\begin{equation}\label{Oh3}
q(x)=\tilde  q(x) +O(h^3),
\end{equation}
and by (\ref{yq})
\begin{equation}\label{q2gOh3}
q(x)=g(x) +O(h^3).
\end{equation}

\subsection{Enhanced global curve approximation}\label{Egca}\hfill

\medskip
Using the above procedure for local $O(h^3)$ approximations to $\Gamma$, we can now improve upon the global $O(h)$ approximation presented above in Section \ref{AppGamma}, as follows:

\begin{enumerate}
\item Denote the union of all the locally computed quadratic approximation curves to the different sections of $\Gamma$ by $G$.
\item Overlay on $[0,1]^2$ a mesh of $n\times n$ points $Q$.
\item Divide $Q$ into two sets of points, $Q_1$ and $Q_2$, separated by the points in $G$.
\item Attach to each point $p\in Q_1$ the value $dist(p,G)$ and to each point $p\in Q_2$ the value $-dist(p,G)$.
\item Find a bicubic spline $S_3$ on $[0,1]^2$ approximating the above data on $Q_1$ and $Q_2$ using a local quasi-interpolation operator.
\item The zero level curve of this spline, $\tilde \Gamma$, is the suggested enhanced approximation to $\Gamma$.

As we show below, $\tilde \Gamma$ is essentially an $O(h^3)$ approximation of $\Gamma$.

\end{enumerate}

\subsubsection{Curve approximation analysis}\label{Curveappanal}\hfill

A similar approximation quality $\tilde\Gamma\sim\Gamma$ is obtained if we define the spline $S_3$ by quasi-interpolation to the signed-distance values defined on $Q$.
In the following proposition, we analyze the quality of the approximation $\tilde\Gamma\sim\Gamma$, obtained by quasi-interpolation, for a class of smooth singularity curves. Below we cite the basic approximation order result from \cite{ALR}.
\begin{proposition}\label{gammaprop}
Assume $\Gamma$ is a $C^4$ smooth curve in $[0,1]^2$ with minimal curvature radius $>R$, and such that its $R$-neighborhood is not self-intersecting and not intersecting the boundaries of $[0,1]^2$ . Let $Q$ be a mesh of $n\times n$ points in  $[0,1]^2$,  of mesh size $d<R/10$. Let $\Gamma$ subdivide $[0,1]^2$ into two domains $\Omega_1$ and $\Omega_2$, and separate the mesh points $Q$ into $Q_1$ and $Q_2$ respectively. Let us attach signed-distance values to the mesh points $Q$ as follows:
$$D(p)=dist(p,\Gamma), \ \ \ p\in Q_1,$$
$$D(p)=-dist(p,\Gamma), \ \ \ p\in Q_2.$$
Let $A$ be the quasi-interpolation operator from
sets of values on $Q$ to the space of bi-cubic splines on $[0,1]^2$ with uniform mesh of mesh size $d$, and assume $\|A\|_\infty<\beta$.
Let $S_\epsilon$ be the bi-cubic spline defined by applying $A$ to the perturbed signed-distance data
$$\tilde{D}(p)=D(p)+\epsilon_p,\ \ \ |\epsilon_p|\le \epsilon ,$$
and let $\Gamma_\epsilon$ be the zero-level curve of $S_\epsilon$. Then
\begin{equation}\label{disteps}
dist(\Gamma,\Gamma_\epsilon)\le  C_0d^4+\beta\epsilon.
\end{equation}
\end{proposition}

\begin{corollary}\label{gammacor}
Let the singularity curve $\Gamma$ satisfy the assumptions in Proposition \ref{gammaprop}.
Let the approximate signed-distance values be defined as in Section \ref{Egca} with $n=O(h^{-\frac{3}{4}})$, with $S_3$ defined
using a cubic quasi-interpolation approximation operator, and the resulting approximated curve denoted by $\tilde\Gamma$. Then, as $h\to 0$,
\begin{equation}\label{diste}
dist(\Gamma,\tilde\Gamma)=O(h^3).
\end{equation}
\end{corollary}
\begin{proof}
By the construction in Section \ref{Egca}, and using (\ref{q2gOh3}), it follows that the perturbation parameter in Proposition \ref{gammaprop} satisfies
$$\epsilon =O(h^3).$$
Choosing $n=O(h^{-\frac{3}{4}})$ it follows that $d=O(h^{\frac{3}{4}})$,
and the result follows by (\ref{disteps}).

\end{proof}

\section{Reconstructing $f$}

Now that we have a third-order approximation to the singularity curve we are ready to construct the approximation to $f$. We have tested two approximation strategies.

\subsection{Approximation by fitting cell-averages}\hfill

\medskip
One obvious strategy is by fitting the cell-averages of the approximant to the given cell-average data, at all cells. The first step is to use the local high third-order approximation for generating a high order global approximation of $\Gamma$. We suggest doing it as described in Section \ref{AppGamma} above, i.e., finding a bi-cubic spline $S_3(x,y)$ approximating the signed-distance from $\Gamma$. Here the signed-distance data is computed using the piecewise quadratic approximations to $\Gamma$ discussed in the previous section. Another important rule of the resulting (implicit) approximation of the singularity curve is that it provides a simple characterization of the corresponding approximation of the two subdomains $\Omega_1$ and $\Omega_2$:
\begin{equation}\label{Omegaappr}
\begin{split}
\tilde\Omega_1=\{(x,y)| S_3(x,y)>0\}\cap [0,1]^2, \\
\qquad \tilde\Omega_2=\{(x,y)| S_3(x,y)<0\}\cap [0,1]^2.
\end{split}
\end{equation}

We consider approximating $f_1$ on $\tilde{\Omega}_1$ and
we consider approximating $f_2$ on $\tilde{\Omega}_2$ by a linear combination of basis functions,
$$\{p_i\}_{i=1}^{M_1+M_2}=\{p_i^{[1]}\}_{i=1}^{M_1}\cup \{p_i^{[2]}\}_{i=1}^{M_2},$$
\begin{equation}\label{f12eqsum}
\begin{split}
f_1(x,y)\sim \sum_{i=1}^{M_1} a^{[1]}_ip^{[1]}_i(x,y)\equiv P^{[1]}(x,y), \\
\qquad
f_2(x,y)\sim \sum_{i=1}^{M_2} a^{[2]}_ip^{[2]}_i(x,y)\equiv P^{[2]}(x,y).
\end{split}
\end{equation}

For example, we may use two sets of bi-cubic spline basis functions, one with B-splines restricted to $\tilde\Omega_1$ and the other with B-splines restricted to $\tilde\Omega_2$.

It is convenient and worthwhile to consider the approximant in the form
$$\tilde f(x,y)\equiv P^{[1]}(x,y){\bf 1}_{\tilde\Omega_1}(x,y)+P^{[2]}(x,y){\bf 1}_{\tilde\Omega_2}(x,y).$$

Let us denote by ${\bf{\bar p}}_i^{[j]}$ the vector of cell-averages of $p_i^{[j]}$, $j=1,2$. A good approximation to $f$ implies a good approximation to the vector of cell-averages $\bf{\bar f}$ of $f$:
\begin{equation}\label{barP}
{\bf{\bar f}}\sim \sum_{i=1}^{M_1} a_i^{[1]}{\bf{\bar p}}_i^{[1]}+\sum_{i=1}^{M_2} a_i^{[2]}{\bf{\bar p}}_i^{[2]}\equiv \bf{\bar P},
\end{equation}
and vice versa, a good fitting of the cell averages induces a good approximation of the function. Therefore, we define the approximation to the function by solving the system of equations

\begin{equation}\label{systembarf}
\sum_{i=1}^{M_1} a_i^{[1]}{\bf{\bar p}}_i^{[1]}+\sum_{i=1}^{M_2} a_i^{[2]}{\bf{\bar p}}_i^{[2]}=\bf{\bar f},
\end{equation}
for the coefficients $\{a_i^{[1]}\}_{i=1}^{M_1}$, $\{a_i^{[2]}\}_{i=1}^{M_2}$.

We are given $N^2$ cell-average data values, while the number of unknowns is $M_1+M_2$ may be different, usually smaller, than $N^2$. We suggest solving the system in the least-squares sense, and the challenge is to analyze the quality of the resulting approximation. We notice that an $O(h^3)$ deviation in the position of $\Gamma$ implies an $O(h^2)$ deviation in the computation of the cell-averages in cells containing the singularity curve. Hence, we cannot expect more than an $O(h^2)$ approximation to $f$ using the above strategy.

\subsection{Approximation of $f$ using quasi-interpolation}\hfill

\medskip
Two ways for directly approximating a function from cell-average values are presented in \cite{ALRY}. The first one is based on Taylor expansions and its advantage approach is the simplicity of both its application and its analysis. The second method is founded on quasi-interpolation theory characterized by its flexibility. In this subsection, we briefly review this technique adapting it from the general form to two dimensions (see more details in \cite{ALRY}).

Let start denoting by $p\in\mathbb{N}$  and for all $(x,y)\in \mathbb{R}^2$:
$$B_{p}\left(x,y\right)=B_{p}(x)B_{p}(y)$$
where $B_{p}$ is the $p$-degree B-spline supported on $I_{p} =\left[-\frac{p+1}{2},\frac{p+1}{2}\right]$, with equidistant knots $
S_{p} = \left\{-\frac{p+1}{2},\hdots,\frac{p+1}{2}\right\}$. Given the values $\bar f_{i,j}$ in (\ref{2DAC}) and $\mathbf{n}\in \mathbb{N}^2,$ we denote as
$$\bar{f}_{\mathbf{n},p}=\left\{\bar{f}_{n_{1}+j_1,n_{2}+j_2}:0\leq |j_l|\leq \left\lfloor \frac{p}{2}\right\rfloor, l=1,2\right\}.$$
where $\lfloor \cdot \rfloor$ is the floor function and consider the linear operator
\begin{equation}\label{operadorLmulti}
L_{p}(\bar{f}_{\mathbf{n},p})=\sum_{j_1,j_2=-\left\lfloor \frac{p}{2} \right\rfloor}^{\left\lfloor \frac{p}{2} \right\rfloor} c_{p,j_1}c_{p,j_2} \bar{f}_{n_1+j_1,n_2+j_2},
\end{equation}
where the coefficients $c_{p,j}$, $j=-\left\lfloor \frac{p}{2} \right\rfloor,\hdots,\left\lfloor \frac{p}{2} \right\rfloor$ are:
\begin{equation}\label{coeficientesc}
c_{p,j}=\sum_{l=0}^{\left\lfloor \frac{p+1}{2} \right\rfloor-1}\frac{t(2l+p+1,p+1)}{{{2l+p+1}\choose{p+1}}}\sum_{i=0}^{2l}\frac{(-1)^i}{i!(2l-i)!}\delta_{l-i+\left\lceil\frac{p+1}{2}\right\rceil,j+1+\left\lfloor\frac{p}{2}\right\rfloor},
\end{equation}
 $\delta_{i,j}$ the Kronecker delta and $t(i,j)$ are the central factorial numbers of the first kind (see \cite{speleers} and \cite{butzeretal}) which can be computed recursively as:
\begin{equation*}
t(i,j)=\left\{
         \begin{array}{ll}
           0, & \hbox{if}\,\, j>i, \\
           1, & \hbox{if}\,\, j=i, \\
           t(i-2,j-2)-\left(\frac{i-2}{2}\right)^2t(i-2,j), & \hbox{if} \,\, 2\leq j <i,
         \end{array}
       \right.
\end{equation*}
with
$$t(i,0)=0,\quad t(i,1)=\prod_{l=1}^{i-1}\left(\frac{i}{2}-l\right),\,\, i\geq 2,$$
and $t(0,0)=t(1,1)=1, t(0,1)=t(1,0)=0$. We show some values of $c_{p,j}$ in Table \ref{tablaCs}.

\begin{table}[h]
  \begin{center}
   \begin{tabular}{crrrrrrr}\hline
                             & $c_{1,j}$ & $c_{2,j}$ & $c_{3,j}$ & $c_{4,j}$& $c_{5,j}$   \\   \hline
                 $ j=0$      &  $1$  & $5/4 $    & $4/3$    & $319/192$  & $73/40$     \\
                 $ j=1$      &       & $-1/8 $   & $-1/6$   & $-107/288$ & $-7/15$     \\
                 $ j=2$      &       &           &          & $47/1152$  & $13/240$    \\ \hline
   \end{tabular}
\medskip
    \caption{$c_{p,j}$ values for $p=1,\hdots,5$. They are symmetric, $c_{p,j}=c_{p,-j}$, $j=0,\hdots,\left\lfloor\frac{p}{2}\right\rfloor$.}
    \label{tablaCs}
  \end{center}
\end{table}

With these ingredients, in \cite{ALRY}, it is proved that  if $f$ is $C^{p+1}$ smooth, the operator
\begin{equation}\label{Qqmulti}
\mathcal{Q}_{p}(\bar{f})(x,y)=\sum_{\mathbf{n}\in\mathbb{Z}^2} L_{p+1}(\bar{f}_{\mathbf{n},p+1})B_{p}\left(\frac{x}{h}-n_1,\frac{y}{h}-n_2\right)
\end{equation}
approximates $f$ with order of accuracy $O(h^{p+1})$.
\begin{equation}\label{Qerror}
|f(x,y)-\mathcal{Q}_{p}(\bar{f})(x,y)|\le Ch^{p+1}.
\end{equation}

In the present context, we are going to use the above approximation operator with $p=3$, aiming at a global $O(h^4)$ approximation. Hence, we shall provide two bicubic spline approximations, one on $\tilde{\Omega}_1$ and the other on $\tilde{\Omega}_2$. Considering the operator $\mathcal{Q}_{3}$, in view of the definition of
$L_{3}$ and of the support of the cubic B-spline, it follows that the computation of $\mathcal{Q}_{3}(\bar{f})(x,y)$ at a given point $(x,y)$ requires the values of the cell-average values at $7\times 7$ cells surrounding that point. This presents a problem if the point $(x,y)$ is near the boundary of the domain, or near the singularity curve $\Gamma$. Such a situation is displayed in Figure \ref{extend}, where the approximation at the point marked by ${\bf x}$ requires cell-average data at all the $7\times 7$ cells within the square black frame. Some of these cells are on the other side of the singularity curve. Below we suggest a constructive approach for solving this problem, assuming the following restrictions hold:

W.L.O.G., let us consider a point $(x,y)\in\tilde{\Omega}_1$, and assume that one of the cell-average values required for computing $\mathcal{Q}_{3}(\bar{f_1})(x,y)$ is either missing since $(x,y)$ is near the boundary of $[0,1]^2$, or, it is not a cell-average of $f_1$ since the cell intersects $\tilde{\Omega}_2$. We refer to such cells as non-valid cells, distinguishing from valid cells which are included in $\tilde{\Omega}_1$.

We shall re-define the value of a non-valid cell by a simple extrapolation of valid cell-average values $\bar f_1$, as follows: Considering such a non-valid cell, we look for the valid cell which is horizontally or vertically the closest to it. In Figure \ref{extend} we assume $\tilde{\Omega}_1$ is to the left of the curve, and the two cells marked by black dots are thus non-valid. Corresponding to the chosen direction, we interpolate four consecutive values of valid cell-averages (marked by red dots) by a cubic polynomial and use it for approximating the value of the cell-average at the non-valid cell. Assuming $f_1$ is $C^4$, the resulting extrapolated cell-average value approximates the exact cell average of the extension of $f_1$ within an $O(h^4)$ accuracy.

\begin{figure}[!ht]
    \includegraphics[width=3in]{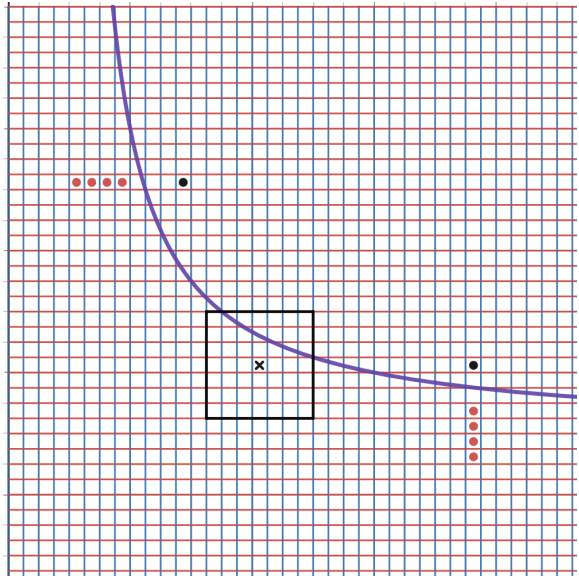}
    \caption{Function extension algorithm display}
    \label{extend}
\end{figure}

We summarize the algorithm in two steps:
\begin{itemize}
\item Firstly, we smoothly extend the cell-average data from $\tilde{\Omega}_1$ and from $\tilde{\Omega}_2$.
\item We use the quasi-interpolation operator $\mathcal{Q}_{3}$ in each part to calculate the approximation of the functions $f_1$ and $f_2$.
\end{itemize}
The resulting approximation is defined as:
\begin{equation}\label{final}
\tilde f(x,y)\equiv \mathcal{Q}_{3}(\bar{f}_1)(x,y){\bf 1}_{\tilde\Omega_1}(x,y)+\mathcal{Q}_{3}(\bar{f}_2)(x,y){\bf 1}_{\tilde\Omega_2}(x,y).
\end{equation}

We note that the approximation rate  $O(h^{p+1})$ of $\mathcal{Q}_{p}$ in (\ref{Qerror}) remains unchanged if the values of the cell-averages used for calculating $L_{p+1}$ contains $O(h^{p+1})$ errors. Therefore, considering $f_1$ and $f_2$ to be the extensions assumed in Assumption \ref{extension}, it follows that:

\begin{equation}\label{Q3error1}
|\big(f_1(x,y)-\mathcal{Q}_{3}(\bar{f}_1)(x,y)\big){\bf 1}_{\tilde\Omega_1}(x,y)|\le Ch^{4},
\end{equation}

\begin{equation}\label{Q3error2}
|\big(f_2(x,y)-\mathcal{Q}_{3}(\bar{f}_2)(x,y)\big){\bf 1}_{\tilde\Omega_2}(x,y)|\le Ch^{4}.
\end{equation}

In the same way, we can construct an approximation of the function $f_1$ or $f_2$ for any order of accuracy $p+1$ using the operator $\mathcal{Q}_p$, Equation \eqref{Qqmulti}, always taking into account the Assumption \ref{extension}.

In view of the above approximation results, together with Corollary \ref{gammacor}, we conclude the following compact result:

\begin{corollary}\label{dHcor}
Let $G(f)$ be the graph of $f$ and let $G(\tilde f)$ be the graph of the approximation in (\ref{final}), then
\begin{equation}
d_H(G(f), G(\tilde f))=O(h^3),
\end{equation}
as $h\to 0$, where $d_H$ is the Hausdorff distance.
\end{corollary}

\section{Conclusions}

In this paper, we have presented a global and explicit reconstruction of a piecewise smooth function in two dimensions from cell-averages data. The main contributions, apart from the improvement in the approximation order, are the global and explicit nature of the resulting approximation.

The reconstruction is adapted to the presence of a smooth jump singularity curve $\Gamma$. We have proposed a global order three approximation of $\Gamma$. The principal ingredients are a signature-based detection with third-order local approximations which are exact for piecewise linear functions with quadratic edges. The local approximations are then used to define an implicit global third-order approximation of $\Gamma$.
Furthermore, new approximation operators are used for defining global explicit spline approximations to the smooth parts of the function directly from the cell-average data.

Our curve approximation results apply for smooth singularity curves and non-vanishing jumps in the function. The more general case, of curves with corners and possible zero jumps at some points along a singularity curve, is a challenging future project.


\clearpage

\section*{Acknowledgements}
The first and third author have been supported through project 20928/PI/18 (Proyecto financiado por la Comunidad Aut\'onoma de la Regi\'on de Murcia a trav\'es de la convocatoria de Ayudas a proyectos para el desarrollo de investigaci\'on cient\'ifica y t\'ecnica por grupos competitivos, incluida en el Programa Regional de Fomento de la Investigaci\'on Cient\'ifica y T\'ecnica (Plan de Actuaci\'on 2018) de la Fundaci\'on S\'eneca-Agencia de Ciencia y Tecnolog\'ia de la Regi\'on de Murcia) and by the national research project PID2019-108336GB-I00. The fourth author
has been supported by grant MTM2017-83942 funded by Spanish MINECO and by grant PID2020-117211GB-I00 funded by MCIN/AEI/10.13039/501100011033.

\end{document}